\documentclass[12pt]{amsart}
\pdfoutput=1
\usepackage{latexsym}
\usepackage[centertags]{amsmath}
\usepackage{amsfonts}
\usepackage{amssymb}
\usepackage{amsthm}
\usepackage{newlfont}
\usepackage{graphics}
\usepackage{color}
\usepackage{wrapfig}
\usepackage[latin1]{inputenc}
\usepackage[usenames,dvipsnames]{xcolor}
\usepackage{graphicx}

\usepackage{wrapfig}

\newtheorem{theo}{Theorem}

\newtheorem{lem} [theo]{Lemma}
\newtheorem{cor}[theo]{Corollary}
\newtheorem{prop}[theo]{Proposition}

\newtheorem{ques}[theo]{Question}

\makeatletter \@addtoreset{equation}{section}
\@addtoreset{theo}{}\makeatother

\def\N{\mathbb{N}}
\def\Z{\mathbb{Z}}

\def\R{\mathbb{R}}%ÓÒ²à
\def\a{\alpha}
\def\b{\beta}
\def\x{\mathbf{x}}

\title{Enumerating Magic Distinct Labellings of the Cube}

\author{Guoce Xin$^{1,*}$, Yingrui Zhang$^2$ and Zihao Zhang$^3$}

 \address{ $^{1,2,3}$School of Mathematical Sciences, Capital Normal University,
 Beijing 100048, PR China}
 \email{$^1$\texttt{guoce\_xin@163.com}\ \& $^2$\texttt{zyrzuhe@126.com} \&  $^3$\texttt{zihao-zhang@foxmail.com} }
\date{July 9, 2021}

\begin{document}

\begin{abstract}
We find by applying MacMahon's partition analysis that all magic labellings of the cube are of eight types, each generated
by six basis elements. A combinatorial proof of this fact is given. The number of magic labellings of the cube is thus reobtained
as a polynomial in the magic sum of degree $5$. Then we enumerate magic distinct labellings, the number of which turns out to be a quasi-polynomial of period 720720. We also find the group of symmetry can be used to significantly simplify the computation.
\end{abstract}

\maketitle

\noindent
\begin{small}
 \emph{Mathematic subject classification}: Primary 05A19; Secondary 11D04; 05C78.
\end{small}
%\subjclass[2010]{Primary 15A15; Secondary 05A15, 11B83}

\noindent
\begin{small}
\end{small}

\section{Introduction}

Let $G$ be a finite graph. A labelling of $G$ is an assignment of a nonnegative integer (in $\N$)
to each edge of $G$. A magic labelling of $G$ with magic sum $r$  is a labelling such that
for each vertex $v$ of $G$ the sum of the labels of all edges incident to $v$ is equal to $r$ (loops are counted as incident only once). It is called a magic distinct labellings if all labels are distinct.
Graphs with a magic labelling are
also called magic graphs.  Magic graphs were studied in great detail by Stanley in \cite{Stanley-Linear-M_G,Stanley-M_G-S_M_G} and Stewart in \cite{Stewart-M_G,Stewart-S_c_g}.
%\cite{Stanley-Linear-M_G,Stanley-M_G-S_M_G,Stewart-M_G,Stewart-S_c_g}\cite{}\cite{}\cite{}.

In this paper, we study magic distinct labellings of the cube. The main difficulty lies in the ``distinct" condition on the labels.
The number of magic labellings of the cube has been enumerated to be a polynomial in the magic sum of degree $5$ \cite{Ahmed-PhD-thesis}, but the distinct case
turns out to be a quasi-polynomial of period 720720. Our starting point is a simple structure
result for magic labellings, from which we are able to extract information about magic distinct labellings.

The cube is a well-known object with 8 vertices, 12 edges (and 6 faces which is irrelevant here), and
each vertex has degree $3$. See Figure \ref{fig:cube}, where we use numbers \{1,...,12\} to mark the 12 edges.
Assume the $i$-th edge has label $x_i$ for $1 \leq i \leq 12$. The labelling is magic
of magic sum $r$ if and only if
%for each vertex $v$ the sum of the labels of the three edges connected with $v$ is equal to $r$. %$H_{cube}(r)$ is also derived in [55]......
%\text{\color{red}{from M. Ahmed, Algebraic Combinatorics of Magic Squares, Ph.D. thesis   }}
\begin{figure}[!ht]
  $$
 %\hskip -1.8in \hskip 1.25in \oD=
 \hskip 0.05in\vcenter { \includegraphics[height= 1.56 in]{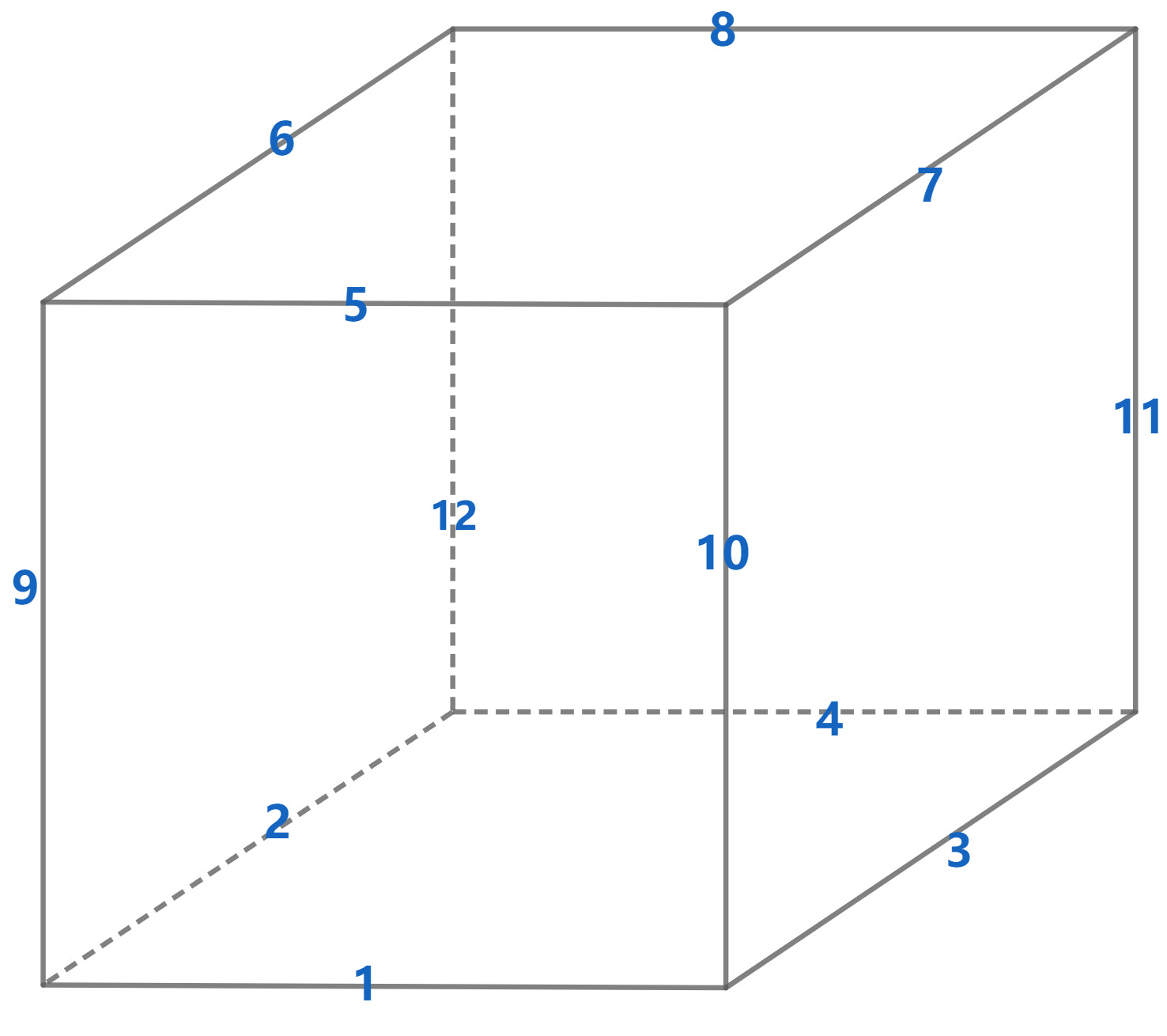}}
%\hskip -3.0in  \vcenter{ \includegraphics[height=1.4 in]{model3.png}}
$$
\caption{ A cube with its edges marked.
\label{fig:cube}}
\end{figure}
$(\a,r)=(x_1,...,x_{12},r)\in \N^{13} $ is a solution to the following linear system.
\begin{align}\label{e-equation}
  \left\{
\begin{array}{cc}
  \begin{array}{rr}
    x_1+x_2+x_9=r & x_1+x_3+x_{10}=r\\
    x_2+x_4+x_{12}=r & x_3+x_4+x_{11}=r \\
    x_5+x_6+x_9=r & x_5+x_7+x_{10}=r\\
     x_6+x_8+x_{12}=r & x_7+x_8+x_{11}=r\\
  \end{array} & ,
\end{array}
\right.
\end{align}
Denote by $S_\R$ the set of $\R$ solutions of the linear system \eqref{e-equation},
and by $S_\Z=S_\R\cap \Z^{13}$ and similar for $S_\N$. Then $S_\N$ is just the set of magic labellings. Since $r$ is determined by $\a$, we will omit the $r$ by saying that $S_\R\subseteq \R^{12}$ when clear from the
context.

By linear algebra, it is easy to see that $S_\R$ is a subspace of $\R^{12}$ of dimension $6$.
There would be little problem to describe $S_\Z$, since it forms a subgroup of $\mathbb{Z}^{12}$ and is hence a finitely-generated free
abelian group. The number of generators of $S_\mathbb{Z}$ is equal to $6$, and there are well-known algorithm for
finding the generators of $S_{\mathbb{Z}}$ explicitly. The structure of $S_\N$ is not so clear: it is only a monoid (semi-group with a unit).

\begin{ques}\label{q-question}
Can we decompose $S_\N$ into some shifted free monoids?  That is, can we find sets $S_{\N}^1, \dots, S_{\N}^t$,
each of the form
$$S_{\N}^i=\{\a|\a=T_i+k_1\b_{i_1}+\dots+k_6\b_{i_6}, T_i \in S_{\N}, k_j \geq 0, \b_{i_j} \in S_{\N} \text{ are linearly independent} \},$$
such that
$S_{\N}=\biguplus_{i=1}^t S_{\N}^i$, where $\biguplus$ means disjoint union.
\end{ques}

As a reward for getting question \ref{q-question}, it is better to use generating functions. Define
$$G(\x)=G(x_1,\dots,x_{12})=\sum_{\alpha\in S_\N} \x^\a,$$
where if $\alpha=(a_1,\dots,a_{12})$ then $\x^\a=x_1^{a_1}\cdots x_{12}^{a_{12}}$. And define
$G_i(\x)$ similarly for $S_{\N}^i$. Then the decomposition give rise the following.
$$G(\x)=\sum_{i=1}^t G_i(x)= \sum_{i=1}^t \frac {\x^{T_i}} {\prod_{j=1}^6(1-\x^{\b_{i_j}})}, \text{ for } 1 \leq i \leq t,$$
where each $G_i(\x)$ corresponds to a shifted free monoid.

In order to answer question \ref{q-question}, we use unit vectors $e_i \in \N^{12}(1\leq i \leq 12)$
that has digit $1$ at the $i$-th position and has digit $0$ at all the others. For example, $\a=(1,0,0,2,0,1,0,0,0,1,0,0)$ can be
written as $\a=e_1+2e_4+e_6+e_{10}$.

Let
\begin{align}\label{b-base}
  \begin{array}{lll}
    \a_1=e_2+e_3+e_5+e_8, & \a_2=e_1+e_4+e_6+e_7, & \a_3=e_1+e_5+e_{11}+e_{12}, \\
    \a_4=e_2+e_6+e_{10}+e_{11}, & \a_5=e_1+e_4+e_5+e_8, & \a_6=e_4+e_8+e_9+e_{10},
  \end{array}
\end{align}
\begin{align}\label{b-basemore}
  \begin{array}{lll}
    \a_7=e_3+e_7+e_9+e_{12}, & \a_8=e_2+e_3+e_6+e_7, & \a_9=e_9+e_{10}+e_{11}+e_{12},
  \end{array}
\end{align}
\begin{align}\label{b-baseT}
 T=\sum_{i=1}^{12}e_i.
\end{align}
Observe that each $\alpha_i$ corresponds to a perfect matching.
These vectors have the following relations.
\begin{align*}
  \a_1&=\a_3 +\a_4 +\a_6 +\a_ 7-\a_2-2\a_9,\\
  \a_2&=\a_3 +\a_4 +\a_6 +\a_7 -\a_1 -2\a_9, \\
  \a_3&=\a_1 +\a_2 +\a_4 +\a_7 -\a_6 -2\a_8 ,\\
  \a_4& = \a_1 + \a_2+ \a_3 + \a_6- 2\a_5- \a_7,\\
  \a_6&=\a_1+\a_2 +\a_4 +\a_7 - \a_3 -2\a_8 ,\\
  \a_5&=\a_1 +\a_2 -\a_ 8=\a_3 +\a_ 6-\a_ 9= \frac{1}{2} (\a_1+\a_2+\a_3+\a_6-\a_4-\a_7), \\
  T&=\a_1+\a_2+\a_9=\a_3+\a_6+\a_8=\a_4+\a_5+\a_7=\frac{1}{2}(\a_1+\a_2+\a_3+\a_4+\a_6+\a_7).
\end{align*}

%extreme rays play an important role. For magic labelling of the cube,
%they are
%
%
%\begin{prop}
%There exist linearly independent vectors $\a_1, \dots, \a_6$ with $a_i \in S_{\N}$ for all $\a \in S_{N}$ such that
%$\a=q_1\a_1+\cdots+q_6\a_6$ where $q_i \in \mathbb{Q}(1 \leq i \leq 6)$.
%\end{prop}

Now, we give the structure result as follows.

\begin{theo}\label{t-main-result}
Every magic labelling of the cube, i.e., $\a \in S_{\N}$, can be written uniquely in one of
the following eight types.
\begin{enumerate}
  \item[t1:] $k_1\a_1+k_2\a_2+k_3\a_3+k_4\a_4+k_5\a_5+k_6\a_6$,
  \item[t2:] $\a_7+k_1\a_1+k_2\a_2+k_3\a_3+k_4\a_5+k_5\a_6+k_6\a_7$,
  \item[t31:] $\a_9+k_1\a_2+k_2\a_3+k_3\a_4+k_4\a_6+k_5\a_7+k_6\a_9$,
  \item[t32:]  $(\a_1+\a_9)+k_1\a_1+k_2\a_3+k_3\a_4+k_4\a_6+k_5\a_7+k_6\a_9$,
  \item[t33:]  $(\a_6+\a_8)+k_1\a_1+k_2\a_2+k_3\a_4+k_4\a_6+k_5\a_7+k_6\a_8$,
  \item[t34:] $\a_8+k_1\a_1+k_2\a_2+k_3\a_3+k_4\a_4+k_5\a_7+k_6\a_8$,
  \item[t351:]  $(\a_4+\a_7)+k_1\a_1+k_2\a_2+k_3\a_3+k_4\a_4+k_5\a_6+k_6\a_7$,
  \item[t352:]  $T+k_1\a_1+k_2\a_2+k_3\a_3+k_4\a_4+k_5\a_6+k_6\a_7$,
\end{enumerate}
where $k_i \in \N, \ 1 \leq i \leq 6$ and $\a_1, \dots, \a_9, T$ are as in $(\ref{b-base}),(\ref{b-basemore}), (\ref{b-baseT})$.
\end{theo}

\begin{cor}\label{c-corollary}
The generating function of $S_\N$ can be decomposed as
\begin{align}\label{G-generatingfunction}
  G(\x,y)=\sum_{\a \in S_{\N}}\x^{\a}y^{r(\alpha)}=\sum_{i=1}^8G_i(\x,y),
\end{align}
 where $G(\x)=G(\x,1)$ and
\begin{itemize}
 \item $G_1(\x,y)= \frac 1 {(1-\x^{\a_1}y){(1-\x^{\a_2}y)}{(1-\x^{\a_3}y)}{(1-\x^{\a_4}y)}{(1-\x^{\a_5}y)}{(1-\x^{\a_6}y)}}$
  \item $G_2(\x,y)=   \frac {\x^{\a_7}y} {(1-\x^{\a_1}y){(1-\x^{\a_2}y)}{(1-\x^{\a_3}y)}{(1-\x^{\a_5}y)}{(1-\x^{\a_6}y)}{(1-\x^{\a_7}y)}}$
  \item $G_3(\x,y)= \frac {\x^{\a_9}y} {(1-\x^{\a_2}y){(1-\x^{\a_3}y)}{(1-\x^{\a_4}y)}{(1-\x^{\a_6}y)}{(1-\x^{\a_7}y)}{(1-\x^{\a_9}y)}} $
  \item $G_4(\x,y)= \frac {\x^{\a_1 +\a_9}y^2} {(1-\x^{\a_1}y){(1-\x^{\a_3}y)}{(1-\x^{\a_4}y)}{(1-\x^{\a_5}y)}{(1-\x^{\a_6}y)}{(1-\x^{\a_9}y)}}$
  \item $G_5(\x,y)= \frac {\x^{\a_6 +\a_8}y^2} {(1-\x^{\a_1}y){(1-\x^{\a_2}y)}{(1-\x^{\a_4}y)}{(1-\x^{\a_6}y)}{(1-\x^{\a_7}y)}{(1-\x^{\a_8}y)}}$
  \item $G_6(\x,y)=  \frac {\x^ {\a_8}y} {(1-\x^{\a_1}y){(1-\x^{\a_2}y)}{(1-\x^{\a_3}y)}{(1-\x^{\a_4}y)}{(1-\x^{\a_7}y)}{(1-\x^{\a_8}y)}} $
  \item $G_7(\x,y)= \frac {\x^{\a_4+\a_7}y^2} {(1-\x^{\a_1}y){(1-\x^{\a_2}y)}{(1-\x^{\a_3}y)}{(1-\x^{\a_4}y)}{(1-\x^{\a_6}y)}{(1-\x^{\a_7}y)}}$
  \item $G_8(\x,y)= \frac {\x^Ty^3} {(1-\x^{\a_1}y){(1-\x^{\a_2}y)}{(1-\x^{\a_3}y)}{(1-\x^{\a_4}y)}{(1-\x^{\a_6}y)}{(1-\x^{\a_7}y)}}$.
\end{itemize}
\end{cor}

%$\b_1=(1,1,1,1,1,1,1,1,1,1,1,1)=\a_1+\a_2+\a_9$
%$\b_2=(0,1,1,0,1,0,0,1,1,1,1,1)=\a_1+\a_9$
%$\b_3=\a_9,\b_4=\a_4,\b_5=\a_7,$
%$\b_6=(0,1,1,1,1,0,0,2,1,1,0,0)=\a_1+\a_6$
%
%
%\begin{theo}
%Every magic labeling of cube, up to the action of the symmetry group of the cube, can be written uniquely as $k_1\b_1+k_2\b_2+k_3\b_3+k_4\b_4+k_5\b_5+k_6\b_6$, where $k_i$ for $1 \leq i \leq 6$ are nonnegative integers and $\b_j (1\leq j \leq 9)$ are as in*****
%\end{theo}

The paper is organized as follows. In this introduction, we have introduced the basic
concepts and our main result.
In Section \ref{s-combinatorial}, we give a combinatorial proof of Theorem \ref{t-main-result}.
In Section \ref{sec-Magic-distinct}, we enumerate magic distinct labellings of the cube in two ways using MacMahon's partition analysis.
One way is by direct computation and divide and conquer; the other way is by using group of symmetry to simplify the computation.
Now the enumeration result can be reproduced in 3 minutes.

%****Section 3 proves that dinv sweeps to area, and Section 4 proves
%that area sweeps to bounce. Finally Section 5 gives a summary and a conjecture on q, t
%symmetry.

\section{A Combinatorial Proof for Magic Labellings\label{s-combinatorial}}

In this section, we will give a combinatorial proof of Theorem \ref{t-main-result}. The $T$ and $\a_i$'s are all as in $(\ref{b-base}), (\ref{b-basemore}), (\ref{b-baseT})$.
To prove Theorem \ref{t-main-result}, we need two lemmas for checking if a given
$\a \in S_\R$ belongs to $S_{\N}$ when it is written using the basis $\a_1,\dots, \a_6$ as
\begin{align}\label{alpha-formula}
  \a= q_1 \a_1 +q_2 \a_2+q_3 \a_3+q_4 \a_4+q_5 \a_5+q_6 \a_6, \qquad q_1,\dots, q_6 \in \R.
\end{align}

\begin{lem}\label{l-condition}
For $\a$ satisfying $(\ref{alpha-formula})$, the condition $\a \in S_{\N}$ is equivalent to the following conditions
on the $q_i$'s.
\begin{description}
    \item[C1] $q_1,q_2,q_3,q_6 \in \N;$
    \item[C2] $(q_1+q_4),(q_2+q_4),(q_3+q_4),(q_6+q_4) \in \N;$
    \item[C3] $q_1+q_3+q_5, q_1+q_5+q_6, q_2 +q_3+q_5, q_2+q_5+ q_6 \in \N.$
\end{description}
%Further, $q_4, q_5 \in \mathbf{Z}.$
\end{lem}
\begin{proof}
We have
\begin{align*}
\a=& q_1 \a_1 +q_2 \a_2+q_3 \a_3+q_4 \a_4+q_5 \a_5+q_6 \a_6 \\
  =&(q_2 +q_3+q_5)e_1+ (q_1+q_4)e_2+q_1e_3+(q_2+q_5+ q_6) e_4 +(q_1+q_3+q_5)e_5+(q_2+q_4)e_6 \\
   &+q_2 e_7+(q_1+q_5+q_6)e_8+q_6 e_9+(q_4+q_6)e_{10}+(q_3+q_4)e_{11}+q_3 e_{12}.
\end{align*}
The lemma clearly follows since $\a\in S_\N$ if and only if each coordinate belongs to $\N$.
\end{proof}

\begin{lem}\label{l-condition+}
If $\alpha$ can be written as
\begin{align}
  \alpha= T_0+k_1\alpha_1+k_2\alpha_2+\cdots+k_9\alpha_9,
\end{align}
where $T_0$ can be easily checked to be in $S_\N$ and $k_i\in \N, \ 1\le i\le 9$, then
$\alpha\in S_\N$.
\end{lem}
The lemma clearly follows since $\alpha_i\in S_\N$ for $1\le i \le 9$. Indeed, the $\alpha_i$'s are the
extreme rays of $S_\N$. See \cite{Stanley-vol-1} for further concepts.

%\begin{proof}
%The lemma clearly follows.
%\end{proof}

Now we are ready to outline the idea of the proof
of Theorem \ref{t-main-result}:
According to Lemma \ref{l-condition}, $q_1,q_2,q_3,q_6 \in \N$, so it suffices to characterize $q_4$ and $q_5$.
It is easy to see that $q_4,q_5 \in \Z$. Therefore we need to divide this condition into small pieces so that each piece
has a combinatorial interpretation. Indeed, the types are named after our decompositions.

\begin{proof}[Proof of Theorem \ref{t-main-result}]
Throughout this proof, $\alpha$ is always written as in formula \eqref{alpha-formula}.
By Lemma \ref{l-condition}, we can always assume that $q_1,q_2,q_3,q_6 \in \N$, and $q_4,q_5 \in \Z$.

We first decompose $S_\N$ into small pieces.
By Figure \ref{fig:basic} we divide $S_\N$ into three disjoint parts according to
\begin{description}
  \item[P1] $q_4,q_5 \in \N;$
  \item[P2] $q_4 < 0, q_5 \geq 2q_4;$
  \item[P3] $q_5 <0, q_5 < 2q_4.$
\end{description}
If we denote by $P(S_\N)=\{\alpha\in S_\N: \alpha \text{ satisfies }P\}$, then $S_\N= P1(S_\N)\biguplus P2(S_\N) \biguplus P3(S_\N)$.
%, where
%$\biguplus$ means disjoint union.

\begin{figure}[!ht]
  $$
 \hskip 0.05in\vcenter { \includegraphics[height= 1.96 in]{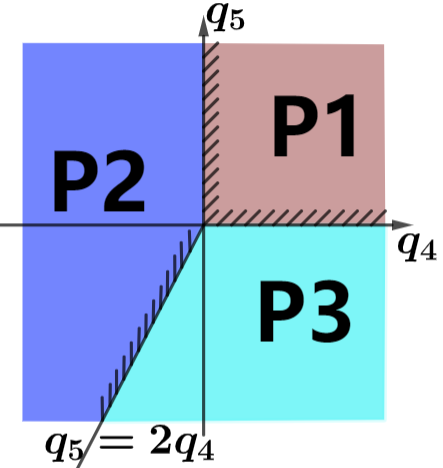}}
$$
\caption{Cut the $(q_4,q_5)$-plane into three disjoint regions.
\label{fig:basic}}
\end{figure}

We further decompose $P3(S_\N)$ into small pieces. Now condition P3 holds true.
By condition C2, it is easy to see that $2q_1+q_5, 2q_2+q_5, 2q_3+q_5$ and $2q_6+q_5 \in \Z$.
%We now divide this condition into small pieces.
We use two rectangular coordinate systems to form a four-dimensional space in Figure \ref{fig:basic1}.
The regions are as follows
\begin{description}
  \item[A1] $2q_1+q_5 \geq 0$ and $2q_2+q_5 \geq 0;$
  \item[B1] $2q_1+q_5 < 0$ or $2q_2+q_5 < 0;$
  \item[A2] $2q_3+q_5 \geq 0$ and $2q_6+q_5 \geq 0;$
  \item[B2] $2q_3+q_5 < 0$ and $2q_6+q_5 < 0.$
\end{description}
%The Figure \ref{fig:basic1} on the left can represent the space $2q_1+q_5, 2q_2+q_5 \in \Z$ and the space $2q_3+q_5, 2q_6+q_5 \in \Z$ on the right.
\begin{figure}[!ht]
  $$
 \hskip -0.95in\vcenter { \includegraphics[height= 1.86 in]{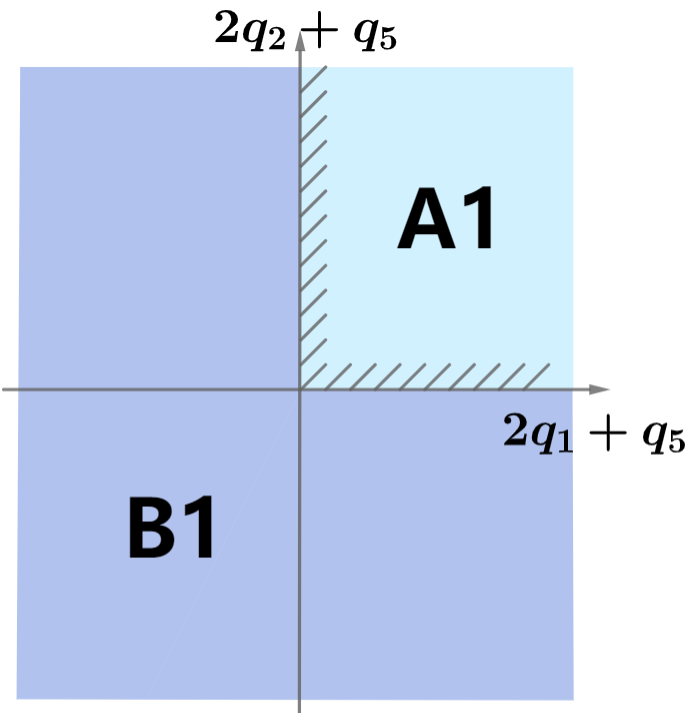}}
 \hskip -4.10in\vcenter { \includegraphics[height= 1.86 in]{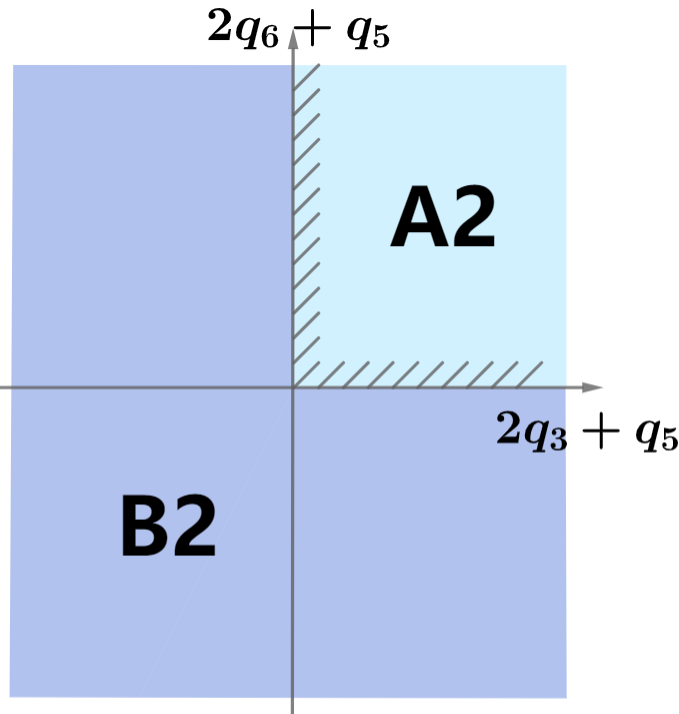}}
$$
\caption{A four-dimensional space: $2q_1+q_5, 2q_2+q_5, 2q_3+q_5, 2q_6+q_5 \in \Z$.
\label{fig:basic1}}
\end{figure}
Now $P3(S_\N)$ is decomposed into four disjoint pieces according to:
\begin{description}
  \item[1] $A1$ and $A2;$
  \item[2] $A1$ and $B2;$
  \item[3] $B1$ and $A2;$
  \item[4] $B1$ and $B2.$
\end{description}
It is easy to check that the last piece is empty by condition C3. For instance,
if $2q_1+q_5<0$ and $2q_3+q_5<0$ then $2(q_1+q_3+q_5)<0$, which is impossible
by condition C3. The second piece $A1$ and $B2$ can be reduced by condition C3
to $A1^{\circ}$ and $B2$ where $A1^{\circ}$ means excluding the equal case in $A1$.
Similarly, the third part $B1$ and $A2$ can be reduced to $B1$ and $A2^{\circ}$
where $A2^{\circ}$ means  excluding the equal case in $A2$.
Then, we only need to consider the following three pieces:
\begin{description}
  \item[1] $A1$ and $A2;$
  \item[2] $A1^{\circ}$ and $B2;$
  \item[3] $B1$ and $A2^{\circ}.$
\end{description}

By slicing $B1$ and $B2$ at the diagonals,
we divide them into five disjoint pieces under the condition P3 according to
%When $2 q_1+q_5<0$ or $2q_2+q_5<0$, and by condition $C3$, we have %$q_5 < 2q_4,
%$2q_3+q_5 > 0$ and $2q_6+q_5 >0$;
%When $2 q_3+q_5<0$ or $2q_6+q_5<0$, and by condition $C3$, we have %$q_5 < 2q_4,
%$2q_1+q_5 > 0$ and $2q_2+q_5 >0$.

%So, the Figure \ref{fig:basic2} can represent the required space with three disjoint subs-spaces.

%\begin{figure}[!ht]
%  $$
% \hskip -1.95in\vcenter { \includegraphics[height= 0.86 in]{basic1.png}}
% \hskip -4.10in\vcenter { \includegraphics[height= 0.86 in]{basic2.png}}
% \hskip -4.20in\vcenter { \includegraphics[height= 0.86 in]{basic3.png}}
%$$
%\caption{ Three disjoint subs-spaces.
%\label{fig:basic2}}
%\end{figure}
\begin{description}
  \item[P31] $2 q_1+q_5<0 $ and $2 q_1+q_5 \leq 2 q_2+q_5$ written as $q_1\leq q_2;$
  %It means  $q_5 <0, q_5 < 2q_4, 2 q_3+q_5 >0, 2 q_6+q_5 >0.$
  \item[P32] $2q_2+q_5<0 $ and $2 q_2+q_5 < 2 q_1+q_5$ written as $q_2 <  q_1;$
  %It means $q_5 <0, q_5 < 2q_4, 2 q_3+q_5 >0, 2 q_6+q_5 >0.$
  \item[P33] $2q_3+q_5<0 $ and $2 q_3+q_5 < 2 q_6+q_5$ written as $q_3 <  q_6;$
  \item[P34] $2 q_6+q_5 <0 $ and  $2 q_6+q_5 \leq 2 q_3+q_5$ written as $ q_6 \leq q_3;$
  %It means $(q_5 <0, q_5 < 2q_4), 2 q_1+q_5 >0, 2 q_2+q_5 >0.$
  %It means  $q_5 <0, q_5 < 2q_4, 2 q_1+q_5 >0, 2 q_2+q_5 >0.$
  \item[P35] $q_5 <0, q_5 < 2q_4, 2q_1+q_5\geq 0, 2q_2+q_5\geq0, 2q_3+q_5\geq 0$ and $2q_6+q_5\geq 0.$
\end{description}

The specific division method is as follows:
1) The first part $A1$ and $A2$ under the condition P3 becomes the condition P35;
2) Slicing of $B2$ by the diagonal gives P33 and P34. On the other hand,
either P33 or P34 will be easily seen to imply the conditions $A1^{\circ}$ and P3;
3) Slicing of $B1$ by the diagonal gives P31 and P32. Similarly, either P31 or P32 will imply  the conditions $A2^{\circ}$ and P3.

%which can be seen in Figure \ref{fig:basic3}.
%In fact, each of conditions $P31,P32,P33,P34$ implies  the condition P3.
%\begin{figure}[!ht]
%  $$
% \hskip -1.65in\vcenter { \includegraphics[height= 1.36 in]{P31.png}}
% \hskip -3.0in  \vcenter{ \includegraphics[height=1.36 in]{P32.png}}
% %\hskip -3.0in  \vcenter{ \includegraphics[height=1.4 in]{P31.png}}
%$$
%
%
% $$
% \hskip -1.65in\vcenter { \includegraphics[height= 1.36 in]{P33.png}}
% \hskip -3.0in  \vcenter{ \includegraphics[height=1.36 in]{P34.png}}
% %\hskip -3.0in  \vcenter{ \includegraphics[height=1.4 in]{P31.png}}
%$$
%
%$$
% \hskip 0.05in\vcenter { \includegraphics[height= 1.36 in]{P35.png}}
% %\hskip -3.0in  \vcenter{ \includegraphics[height=1.36 in]{P34.png}}
% %\hskip -3.0in  \vcenter{ \includegraphics[height=1.4 in]{P31.png}}
%$$
%
%\caption{ Five disjoint parts.
%\label{fig:basic3}}
%\end{figure}

Now we have
$$S_\N= P1(S_\N)\biguplus P2(S_\N) \biguplus P31(S_\N) \biguplus P32(S_\N)
\biguplus P33(S_\N) \biguplus P34(S_\N) \biguplus P35(S_\N).$$

Next we show that each piece corresponds to one type in Theorem \ref{t-main-result}.
Observe that the uniqueness for each type follows by the linear independency of the
$\alpha_i$'s. For instance, in the type t2 case, the linear independency of
 $\a_1,\a_2,\a_3,\a_5,\a_6,\a_7$ can be straightforwardly checked.

\begin{enumerate}
  \item[P1:] When condition Pl holds, simply let $ k_i=q_i$ for all $1\le i\le 6$. Then $k_i \in \N$ and this is exactly the type t1 case.

  \item[P2:] When condition P2 holds, we have $q_5- 2q_4 \in \N$ and $-(q_4+1) \in \N$. Substituting $\a_4 = \a_1 + \a_2+ \a_3 + \a_6- 2\a_5- \a_7$ in formula (\ref{alpha-formula}) gives
\begin{align*}
       \a
       %=&(q_1+q_4) \a_1 + (q_2+q_4) \a_2+(q_3 +q_4) \a_3 + (q_5- 2q_4) \a_5 +(q_6 +q_4)\a_6  - q_4 \a_7 \\
         =\a_7+(q_1+q_4)\a_1 + &(q_2+q_4) \a_2+(q_3 +q_4) \a_3  \\
              &+(q_5- 2q_4) \a_5 +(q_6 +q_4)\a_6  - (q_4+1) \a_7
     \end{align*}
Comparing with the type t2 case, we shall have
$$\begin{array}{lll}
       k_1=(q_1+q_4), & k_2=(q_2+q_4), & k_3=(q_3+q_4), \\
       k_4=(q_5- 2q_4), & k_5=(q_6 +q_4), & k_6=-(q_4+1).
      \end{array}$$
The condition $k_i \in \N$ can be easily checked.

Conversely, if $\alpha$ is as in the type t2 case, then by Lemma \ref{l-condition+} $\alpha\in S_\N$ so that all conditions in Lemma \ref{l-condition} holds.
Then it is easy to check that P2 holds so that $\alpha\in P2(S_\N)$.
The converse part for the other pieces are similar so we will omit that part of the proof.

%Then by condition C2 in Lemma \ref{l-condition}, we have $k_i\in \N$ for all $i$ and this is the type t2 case.
            \item[P31:] When condition P31 holds, we have $-(2 q_1+q_5+1) \in \N$ and $q_2-q_1 \in \N$. Substituting $\a_1=\a_3 +\a_4 +\a_6 +\a_7-\a_2-2\a_9  ,\a_5 =\a_3 +\a_6-\a_9$ in formula (\ref{alpha-formula}) gives
              \begin{align*}
               \a=\a_9+(q_2-q_1)\a_2 +&(q_1+q_3+q_5)\a_3+(q_1+q_4)\a_4 \\
                            &+q_1\a_7+(q_1+q_5+q_6)\a_6 -(2q_1+q_5+1)\a_9.
               \end{align*}
   By comparing with the type t31 case, we shall have
       $$\begin{array}{lll}
       k_1=q_2-q_1, & k_2=q_1+q_3+q_5, & k_3=q_1+q_4, \\
       k_4=q_1, & k_5=q_1+q_5+q_6, & k_6=-(2q_1+q_5+1).
      \end{array}$$
Then $k_i \in \N$ by conditions C1, C2 and C3.

            \item[P32:] When condition P32 holds, we have $-(2q_2+q_5+1)\in \N $ and $q_1-q_2-1 \in \N$.
                Substituting $\a_2=\a_3 +\a_4 +\a_6 +\a_7 -\a_1 -2\a_9 ,\a_5=\a_3 +\a_6 -\a_9$ in formula (\ref{alpha-formula}) gives
             \begin{align*}
            \a=(\a_1+\a_9)+(q_1-q_2-1) \a_1 +&(q_2+q_3+q_5)\a_3  +(q_2+q_4) \a_4 \\
            &+q_2 \a_7+(q_2+q_5+q_6)\a_6 - (2q_2+q_5+1)\a_9
               \end{align*}
    By comparing with the type t32 case, we shall have
             $$\begin{array}{lll}
                   k_1=q_1-q_2-1, & k_2=q_2+q_3+q_5, & k_3=q_2+q_4, \\
                    k_4=q_2, & k_5=q_2+q_5+q_6, & k_6=- (2q_2+q_5+1)
              \end{array}$$

     Then $k_i \in \N$ by conditions C1, C2 and C3.

            \item[P33:] When $2q_3+q_5<0 $ and $q_3 <  q_6 $, we have $-(2q_3+q_5+1) \in \N $ and $q_6 -q_3-1 \in \N $. Substituting
            $\a_3=\a_1 +\a_2 +\a_4 +\a_7 -\a_6 -2\a_8$ and $\a_5 =\a_1 +\a_2 -\a_8$ in formula (\ref{alpha-formula}) gives
             \begin{align*}
              \a=(\a_6+\a_8)+ (q_1+q_3+q_5) \a_1 & +(q_2+q_3+q_5)\a_2+(q_3+q_4) \a_4 \\
              & +(q_6-q_3-1)\a_6+q_3 \a_7 -(2q_3+q_5+1)\a_8
               \end{align*}
       By comparing with the type t33 case, we shall have
             $$\begin{array}{lll}
                   k_1=q_1+q_3+q_5, & k_2=q_2+q_3+q_5, & k_3=q_3+q_4, \\
                    k_4=q_3, & k_5=q_6-q_3-1, & k_6= -(2q_3+q_5+1).
              \end{array}$$

   Then $k_i \in \N$ by conditions C1, C2 and C3.
 % and recall that conditions C1 and C2.Then  type t$33$ can be obtained.

             \item[P34:]  When condition P34 holds, we have $-(2 q_6+q_5+1) \in \N$ and $q_3-q_6 \in \N$. Substituting
              $\a_6=\a_1+\a_2 +\a_4 +\a_7 - \a_3 -2\a_8$ and $\a_5=\a_1 +\a_2 -\a_ 8$ in formula (\ref{alpha-formula}) gives
            \begin{align*}
            \a=\a_8+(q_1+q_5+q_6)\a_1 +&(q_2+q_5+q_6)\a_2 +(q_3-q_6) \a_3 \\
                     &  +(q_4+q_6)\a_4 +q_6 \a_7 -(2q_6+q_5+1)\a_8
            \end{align*}
   By comparing with the type t34 case, we shall have
            $$\begin{array}{lll}
             k_1=q_1+q_5+q_6, & k_2=q_2+q_5+q_6, & k_3=q_3-q_6, \\
             k_4=q_4+q_6, & k_5=q_6, & k_6= -(2q_6+q_5+1).
            \end{array}$$

 Then $k_i \in \N$ by conditions C1, C2 and C3.

%and recall that conditions C1 and C2. Then  type t$34$ can be obtained.

            \item[P35:] When condition $P35$ holds,
            %and $q_5 <0, q_5 < 2q_4.$
              we can substitute  $\a_5= \frac{1}{2} (\a_1+\a_2+\a_3+\a_6-\a_4-\a_7)$  in formula (\ref{alpha-formula}) and get
              \begin{align*}
                \a=\left ( q_1+ \frac{1}{2} q_5 \right)\a_1 &+\left(q_2+\frac{1}{2} q_5 \right)\a_2 +\left(q_3+\frac{1}{2} q_5\right)\a_3  \\
                & +\left(q_4 -\frac{1}{2} q_5\right)\a_4- \frac{1}{2} q_5\a_7 +\left(q_6 +\frac{1}{2} q_5\right)\a_6.
              \end{align*}

              According to the parity of $q_5$, we divide it into two cases:

                    Case 351: When $q_5$ is even,  we rewrite $\a$ as
                    \begin{align*}
                      \a =(\a_7+\a_4)+&\left( q_1+ \frac{1}{2} q_5\right)\a_1 +\left(q_2+\frac{1}{2} q_5\right)\a_2 +\left(q_3+\frac{1}{2} q_5\right)\a_3 \\
                      & +\left(q_4 -\frac{1}{2} q_5-1\right)\a_4 - \left(\frac{1}{2} q_5+1\right)\a_7 +\left(q_6 +\frac{1}{2} q_5\right)\a_6.
                    \end{align*}
            By comparing with the type t351 case, we shall have
                     $$\begin{array}{lll}
                   k_1= q_1+ \frac{1}{2} q_5, & k_2=q_2+\frac{1}{2} q_5, & k_3=q_3+\frac{1}{2} q_5, \\
                    k_4=q_4 -\frac{1}{2} q_5-1, & k_5=- (\frac{1}{2} q_5+1), & k_6=q_6 +\frac{1}{2} q_5.
              \end{array}$$
By condition P35 and the assumption $q_5$ is even, we see that
$k_i \in \N$ for all $i$.
                        %$\a_7+\a_4+( q_1+ \frac{1}{2} q_5)\a_1 +(q_2+\frac{1}{2} q_5)\a_2 +(q_3+\frac{1}{2} q_5)\a_3+(q_4 -\frac{1}{2} q_5-1)\a_4 - (\frac{1}{2} q_5+1)\a_7 +(q_6 +\frac{1}{2} q_5)\a_6 $

                    Case 352: When $q_5$ is odd, we rewrite $\a$ as
                   \begin{align*}
                     \a= \frac{1}{2}\left(\a_1+\a_2+\a_3+\a_4+\a_6 +\a_7\right) &+\left( q_1+ \frac{1}{2} q_5-\frac{1}{2}\right)\a_1 \\
                         +\left(q_2+\frac{1}{2} q_5-\frac{1}{2}\right)\a_2 +&\left(q_3+\frac{1}{2} q_5-\frac{1}{2}\right)\a_3+\left(q_4 -\frac{1}{2} q_5-\frac{1}{2}\right)\a_4 \\
                        &  -\left(\frac{1}{2} q_5+\frac{1}{2}\right)\a_7 +\left(q_6 +\frac{1}{2} q_5-\frac{1}{2}\right)\a_6.
                   \end{align*}
               By comparing with the type t352 case, we shall have
               $$T=\frac{1}{2}(\a_1+\a_2+\a_3+\a_4+\a_6+\a_7),$$
               $$\begin{array}{lll}
                   k_1= q_1+ \frac{1}{2} q_5-\frac{1}{2}, & k_2=q_2+\frac{1}{2} q_5-\frac{1}{2}, & k_3=q_3+\frac{1}{2} q_5-\frac{1}{2}, \\
                    k_4=q_4 -\frac{1}{2} q_5-\frac{1}{2}, & k_5=- (\frac{1}{2} q_5+\frac{1}{2}), & k_6=q_6 +\frac{1}{2} q_5-\frac{1}{2}.
              \end{array}$$
By condition P35 and the assumption $q_5$ is odd, we see that
$k_i \in \N$ for all $i$.
             \end{enumerate}
\end{proof}

\def\diag{\mathop{\mathrm{diag}}}
\def\Oeq{{\mathop{\mathrm{\Omega}}\limits_=}}
\def\Oge{{\mathop{\mathrm{\Omega}}\limits_\geq}}
\section{Enumeration for Magic distinct Labellings of the Cube\label{sec-Magic-distinct}}
Let $S^*_\N$ be the set of magic labellings of the cube.
Then $S_\N^*$ is obtained from $S_\N$ by slicing out $\binom{12}{2}$ hyper planes.
We are interested in the generating function
$$G^*(\x,y)= \sum_{\a \in S^*_\N}   \x^\a y^{r(\a)}.$$
It turns out to be so complicated that we only report the enumeration result. That is,
we will give the generating function for the number $h_r$ of magic distinct labellings of the cube with magic sum $r$:
$$G^*(y)=G^*(1,\dots, 1,y)=\sum_{r\ge 0} h_r y^r.  $$

\subsection{Preliminaries}
MacMahon's Omega operators on formal series are defined by
\begin{align*}
  \Oeq \sum_{k=-\infty}^\infty c_i \lambda^i &= c_0,\qquad
  \Oge \sum_{k=-\infty}^\infty c_i \lambda^i =\sum_{k=0}^\infty c_i.
\end{align*}
They are basic ingredients of MacMahon's
partition analysis, which has been restudied by Andrews and his coauthors in a series of papers starting with \cite{Andrews-Omege-1}.
%(see e.g., \cite{MacM-par-analy}).

MacMahon's idea is to use new variables to replace linear constraints so that many problems can be converted into the constant term
of a special kind of rational functions.
It is standard to express $G(\x,y)$ by the following constant term in $\lambda_1,\dots, \lambda_8$.
\begin{align*}
G(\x ,y)=&\Oeq \frac{1} {\left({1-\lambda_{{1}}\lambda_{{2}}x_{{1}}}\right)\left({1-\lambda_{{3}}\lambda_{{4}}x_{{4}}}\right)
\left({1-\lambda_{{2}}\lambda_{{3}}x_{{3}}}\right)\left({1-\lambda_{{1}}\lambda_{{4}}x_{{2}}}\right)}
\\
 &\times \frac{1} {\left({1-\lambda_{{5}}\lambda_{{6}}x_{{5}}}\right)\left({1-\lambda_{{7}}\lambda_{{8}}x_{{8}}}\right)
 \left({1-\lambda_{{6}}\lambda_{{7}}x_{{7}}}\right)\left({1-\lambda_{{5}}\lambda_{{8}}x_{{6}}}\right)}  \\
   &\times \frac{1} {\left({1-\lambda_{{1}}\lambda_{{5}}x_{{9}}}\right)\left({1-\lambda_{{2}}\lambda_{{6}}x_{{10}}}\right)
   \left({1-\lambda_{{3}}\lambda_{{7}}x_{{11}}}\right)\left({1-\lambda_{{4}}\lambda_{{8}}x_{{12}}}\right)
   ( 1
-\frac{y} {\lambda_{{1}} \lambda_{{2}} \lambda_{{3}} \lambda_{{4}} \lambda_{{5}} \lambda_{{6}} \lambda_{{7}} \lambda_{{8}} }
)}.
\end{align*}
The $\lambda$ variables are usually clear from the context and the whole theory relies on unique series expansion of rational functions.
There are algorithm for evaluating this type of constant terms. Such as the Mathematica package \texttt{Omega} in \cite{Andrews-Omege}.
We use the first named author's Maple packages \texttt{Ell} in \cite{Xin-fast-alg} and (often better) \texttt{CTEuclid} in \cite{Xin-euclid-alg}. See \cite{Xin-fast-alg} for a rigorous description about how the above works in general situation,
i.e., in a field of iterated Laurent series.

Let us describe briefly how we discovered the formula \eqref{G-generatingfunction} in Corollary \ref{c-corollary}.
By using \texttt{CTEuclid}, we can obtain an expression of $G(\x,y)$ as a sum of rational functions quickly.
The normal form of $G(\x,y)$ is given by
\begin{align*}
  G(\x ,y) &=\frac{N(\x,y)}{(1-\x^{\a_1} y)(1-\x^{\a_2}y)\cdots (1-\x^{\a_9}y)}, \quad \text{ where}\\
    N(\x,y)&=  1-{y}^{2}x_{{2}}x_{{3}}x_{{6}}x_{{7}}x_{{9}}x_{{10}}x_{{11}}x_{{12}}-{y}^{2}x_{{1}}x_{{4}}x_{{5}}
   x_{{8}}x_{{9}}x_{{10}}x_{{11}}x_{{12}} \\
   &-{y}^{2}x_{{1}}x_{{2}}x_{{3}}x_{{4}}x_{{5}}x_{{6}}x_{{7}}x_{{8}}
   +{y}^{4}x_{{1}}x_{{2}}x_{{3}}x_{{4}}x_{{5}}x_{{6}}x_{{7}}x_{{8}}{x_{{9}}}^{2}{x_{{10}}}^{2}{x_{{11}}}^{2}{x_{{12}}}^{2} \\
   &+{y}^{4}{x_{{1}}}^{2}x_{{2}}x_{{3}}{x_{{4}}}^{2}{x_{{5}}}^{2}x_{{6}}x_{{7}}{x_{{8}}}^{2}x_{{9}}x_{{10}}x_{{11}}x_{{12}}
    +{y}^{4}x_{{1}}{x_{{2}}}^{2}{x_{{3}}}^{2}x_{{4}}x_{{5}}{x_{{6}}}^{2}{x_{{7}}}^{2}x_{{8}}x_{{9}}x_{{10}}x_{{11}}x_{{12}} \\
   &-{y}^{6}{x_{{1}}}^{2}{x_{{2}}}^{2}{x_{{3}}}^{2}{x_{{4}}}^{2}{x_{{5}}}^{2}{x_{{6}}}^{2}{x_{{7}}}^{2}{x_{{8}}}^{2}{x_{{9}}}^{2}{x_{{10}}}^{2}
  {x_{{11}}}^{2}{x_{{12}}}^{2}.
\end{align*}
Such a formula is usually too complex to use.
Formula \eqref{G-generatingfunction} is desirable and is easy to be verified, but it is unclear how to obtain such a decomposition.
We guessed such a formula by using certain criterion.

By setting $x_i=1$ for all $i$, we obtain
$G(y)=G(\mathbf{1},y)= {\frac {{y}^{3}+3\,{y}^{2}+3\,y+1}{ \left( 1-y \right) ^{6}}}.$
This agree with \cite[p. 73]{Ahmed-PhD-thesis}.

When dealing with magic distinct labellings, we also use
the following two operators:
$$ \diag_{x,y} \sum_{i,j\ge 0} a_{i,j} x^i y^j= \sum_{i\ge 0}a_{i,i} x^i y^i, \quad \diag_{x>y}\sum_{i,j\ge 0} a_{i,j} x^i y^j =    \sum_{i>j \ge 0}a_{i,j} x^i y^j. $$
They can be realized by the Omega operators. See \cite{Zhong} for details. Clearly, they can be naturally extended to
$\diag_{x_1>x_2>\cdots>x_k}$, which is treated as the identity operator $1$ when $k\le 1$.

Let $A$ and $B$ be two disjoint sets. For permutations $\pi \in S_A$ and  $\sigma \in S_B$, a permutation $\tau\in S_{A\cup B}$
is said to be a shuffle of $\pi$ and $\sigma$ if $\tau|A=\pi$ and $\tau|B=\sigma$, where $\tau|C$, called the restriction
of $\tau$ in $C$, is obtained from
$\tau$ by omitting elements not in $C$. Denote by $Sh(\pi,\sigma)$ the set of shuffle permutations of $(\pi,\sigma)$.
For instance, if $\pi=253$ and $\sigma=41$, then $Sh(\pi,\sigma)=\{25341,25431,25413, 24531,24513,24153,42531, 42513,42153,41253\}$.

For $\pi=\pi_1\pi_2\cdots \pi_a\in S_A$ and $\sigma=\sigma_1\sigma_2\cdots \sigma_b\in S_B$, define the linear operator
$$\Pi_{\pi,\sigma} F(\x,y) = \diag_{x_{\pi_1}>x_{\pi_2}> \cdots >x_{\pi_a}} \cdot  \diag_{x_{\sigma_1}>x_{\sigma_2}> \cdots >x_{\sigma_b}} F(\x,y)$$
acting on a formal power series $F(\x,y)$. We allow $\sigma$ to be empty.

\subsection{Direct computation by divide and conquer}
Now we describe how to extract $G^*(y)$ from $G(\x,y)$. Because
it is important to keep the number of factors in the denominator of $F$ small when extracting constant term of $F$,
we prefer using $G_i, \ i=1,\dots, 8$ rather than $G$.

We can compute $G^*(\x,y)$ by the formula
$$G^*(\x,y)= \sum_{\pi\in S_{12}} \Pi_{\pi} G(\x,y) = \sum_{i=1}^8 \sum_{\pi\in \mathfrak{S}_{12}} \Pi_{\pi} G_i(\x,y).$$
But then we need to compute $8\times 12! \approx 3.83\times 10^9$ cases, which is out of question.

Thus we device to compute $G^*(y)$ by divide and conquer.

Firstly we partition the index set $\{1,2,\dots, 12\}$ of the edges into disjoint union of
$A=\{1,4,6,7,9,10,11,12\}$ and $B=\{2,3,5,8\}$.
Note that edges with labels in $B$ form an extreme ray, and edges with labels in $A$ form a Hamiltonian cycle.
Other choices of $A$ and $B$ might work, but that take time to test and we find a better way by using symmetry.

Next we compute $\Pi_{\pi,\sigma} G_i(\x,y)$ for each pair $(\pi,\sigma)\in S_A\times S_B$ and keep only those non-varnishing cases.
It turns out that for $i=1,2,5,6$, $\Pi_{\pi,\sigma} G_i(\x,y)\neq 0,$ only for 160 out of $8!\times 4!=967680$ pairs $(\pi,\sigma)$.
The other cases are also reasonable.

Finally for each non-varnishing case, we compute $\Pi_\tau \Pi_{\pi,\sigma} G_i(\x,y)$ for all $\binom{12}{4}=495$ permutations $\tau\in Sh(\pi,\sigma)$,
setting $x_j=1$ for all $j$,
and take their sum. The final result is $G^*(y)$.
In this way we can compute $G^*(y)$ in 240 minutes.

The decomposition on the permutations is:
$$ S_{12} = \biguplus_{(\pi,\sigma)\in S_{A}\times S_{B}} Sh(\pi,\sigma).$$
This leads to the following formula:
$$G^*(\x,y)= \sum_{\pi\in S_{A}} \sum_{\sigma\in S_B} \sum_{\tau\in Sh(\pi,\sigma)}
\Pi_\tau G(\x,y).$$
Now for each $(\pi,\sigma)$, we can write, for all $\tau\in Sh(\pi,\sigma)$,
$$ \Pi_\tau G(\x,y) = \Pi_\tau   ( \Pi_{\pi,\sigma} G(\x,y)) = \sum_{i=1}^8  \Pi_\tau   ( \Pi_{\pi,\sigma} G_i(\x,y)). $$
Then for each $i$, we can discard those $(\pi,\sigma)$ with $\Pi_{\pi,\sigma} G_i(\x,y)=0$.

\medskip
We remark that in practice, we check the nullity of $\Pi_{\pi,\sigma} G_i(\x,y)\big|_{x_j=1}^{j=1,\dots, 12}$ in stead. The reason
is that $G_i(\x,y)$ has only nonnegative coefficients, and so does $\Pi_{\pi,\sigma} G_i(\x,y)$. This is crucial because
it is not easy to check if a multivariate rational function (written as a sum) is $0$.

\subsection{Simplification using group action}
Denote by $U$ the automorphism group of the cube. It is well-known that $U$ has cardinality $3!\times 2^3=48$. (The symmetry group
of the geometric cube (in $\R^3$) has cardinality $24$.) The group $U$ induces an action on the edges, say $x_1,\dots, x_{12}$.

For any magic labelling of the cube $\alpha=(x_1,x_2,\dots, x_{12})\in S_\N$, we may assume $x_1$ to be the smallest under the action of $U$.
Since the all $1$ labelling is in $S_\N$, we may further assume $x_1=0$ by subtracting $(x_1,x_1,\dots, x_1)$.

Let $U_1$ be the subgroup of $U$ that fixes $x_1$. Then $U_1$ consists of $4$ elements generated by: i) (2,3)(9,10)(6,7)(11,12), corresponding to
flipping along the center of edges $x_1$ and $x_8$ in $\R^3$; ii) (3,10)(4,5)(7,11)(6,12)(2,9), which cannot be realized in $\R^3$. When $U_1$ acting on the edges,
we get  5 orbits: $\{x_1\}, \{x_8\}, \{x_4,x_5\}, \{x_2,x_3,x_9,x_{10}\}, \{x_6,x_7,x_{11},x_{12}\}$. The next lemma says that
the second smallest label can only appear in two of the orbits.

\begin{lem} \label{orbit-1}
Suppose $\alpha=(x_1,\dots, x_{12})\in S_\N$ with
$x_1 <x_i $ and  $x_i = \min \{\ x_2,  \dots,  x_{12}\}$. Then $i \notin \{2,3,9,10\} $ and $ i \neq 8$.
\end{lem}
 \begin{proof}
 Under the action of $U_1$, we only need to check the case $i=2$ and $i=8$.

If $x_2$ is the second smallest, then combing $x_1+x_2+x_9=r$  and $x_5+x_6+x_9=r$ gives $ x_2+x_1 =x_5 + x_6$.
This is a contradiction.

If $x_8$ is the second smallest, by using
\begin{equation*}
 x_5+x_6+x_9= r ,  x_6+x_8+x_{12}=r \quad \text{ and }\quad  x_1+x_2+x_9=r,  x_2+x_4+x_{12}=r,
\end{equation*}
 we  obtain respectively: $x_5+x_9 =x_8+x_{12} $ and $ x_1+x_9 =x_4+x_{12}$.
 Combining them gives $x_8+x_1  =x_4 +x_5$. This is again a contradiction.
 \end{proof}

The next result tells how to choose the representative for $\alpha\in S^*_\N$ under the action of $U$.
\begin{prop}
For any magic distinct labelling of the cube $\alpha\in S^*_\N$, there is a unique $u\in U$
 such that $\alpha$ is transformed by $u$ into $(x_1,\dots, x_{12})$ belongs to one of the following 2 types.

i) $x_1$ is the smallest, $x_6$ is the second smallest.

ii) $x_1$ is the smallest, $x_4$ is the second smallest, and $x_6<x_7$.
\end{prop}
\begin{proof}
We have assumed that $x_1$ is the smallest. Now by fixing $x_1$ we act by the group $U_1$.
By Lemma \ref{orbit-1}, the second smallest label can only belongs to two orbits, thus we can transform by element in $U_1$ so that
either $x_6$ or $x_4$ is the second smallest.

Case 1: $x_6$ is the second smallest. Since the orbit of $x_6$ is of size $4$, the  subgroup $U_{1,6}$ of $U_1$ fixes $x_6$ only contains the
identity element.

Case 2: $x_4$ is the second smallest. Since of orbit of $x_4$ is $2$, the subgroup $U_{1,4}$ of $U_1$ fixes $x_4$ contains two element:
the
identity element and $(2,3)(9,10)(6,7)(11,12)$. In this case, add the condition $x_6 < x_7 $ to make the transformation unique.
\end{proof}

Thus we can compute $G^*(\x,y)$ according to the above two cases.
For the first case, we let

$$F_1(\x,y)=\left(\diag_{x_6>x_1}\diag_{x_2>x_6}\cdots \diag_{x_5>x_6}\diag_{x_7>x_6}\cdots  \diag_{x_{12}>x_6} G(\x,y)\right) \bigg|_{x_1=0} .$$

 By using \texttt{CTEuclid}, we get
\begin{align*}
F_1(\x ,y)&=\frac 1 {\left({1-yx_{{4}}x_{{8}}x_{{9}}x_{{10}}}\right)\left({1-yx_{{3}}x_{{7}}x_{{9}}x_{{12}}}\right)}\\
&\times
\frac{y ^{8}x_{{10}} ^{4}x_{{11}} ^{2}x_{{12}} ^{3}x_{{5}} ^{2}x_{{6}} x_{{7}} ^{2}x_{{8}} ^{4}x_{{9}} ^{5}x_{{2}} ^{3}x_{{3}} ^{4}x_{{4}} ^{2}}
{\left({1-yx_{{9}}x_{{10}}x_{{11}}x_{{12}}}\right)
\left({1-yx_{{2}}x_{{3}}x_{{5}}x_{{8}}}\right) \left({1-{y}^{4}{x_{{2}}}^{2}{x_{{3}}}^{2}x_{{4}}x_{{5}}x_{{6}}x_{{7}}
 {x_{{8}}}^{2}{x_{{9}}}^{2}{x_{{10}}}^{2}x_{{11}}x_{{12}}}\right)}.
\end{align*}

The computation of the second case is similar.
 By using \texttt{CTEuclid}, we get
\begin{align*}
F_2(\x,y)&= \frac 1 {\left({1-yx_{{6}}x_{{2}}x_{{7}}x_{{3}}}\right)\left({1-yx_{{3}}x_{{7}}x_{{9}}x_{{12}}}\right)}\\
&\times
\frac{y ^{8}x_{{12}} ^{3}x_{{11}} ^{2}x_{{10}} ^{3}x_{{9}} ^{4}x_{{8}} ^{3}x_{{7}} ^{3}x_{{6}} ^{2}x_{{5}} ^{2}x_{{4}} x_{{2}} ^{4}x_{{3}} ^{5}} {\left({1-yx_{{9}}x_{{10}}x_{{11}}x_{{12}}}\right)
\left({1-yx_{{2}}x_{{3}}x_{{5}}x_{{8}}}\right)
  \left({1-{y}^{4}{x_{{2}}}^{2}{x_{{3}}}^{2}x_{{4}}x_{{5}}
    x_{{6}}x_{{7}}{x_{{8}}}^{2}{x_{{9}}}^{2}{x_{{10}}}^{2}x_{{11}}x_{{12}}}\right)}.
\end{align*}

Use the method of the previous subsection for $F_1(\x,y)$ and $F_2(\x,y)$ respectively, we can easily obtain $F^*_1(y) $ and $F^*_2(y)$.
Finally, by using the fact $G^*(y)= \frac{F^*_1(y)+F^*_2(y)}{1-y^3}$, we obtain
\begin{align*}
G^*(y)=&\frac{y ^{17}\left({1-y}\right)^{4}\left({1-{y}^{2}}\right)^{2} N}
{\left({1-{y}^{3}}\right)\left({1-{y}^{16}}\right)
\prod_{i=5}^{14}(1-y^i)}\\
=&6\,{y}^{17}+13\,{y}^{18}+34\,{y}^{19}+60\,{y}^{20}+128\,{y}^{21}+199\,
{y}^{22}+331\,{y}^{23}+\cdots ,
\end{align*}
where
{\tiny
\begin{align*}
  N=& 6+37\,y+158\,{y}^{2}+514\,{y}^{3}+1451\,{y}^{4}+3621\,{y}^{5}+8284\,{y}^{6}+17569\,{y}^{7}+35070\,{y}^{8}+66312\,{y}^{9}+119823\,{y}^{10}+207728\,{y}^{11} \\
  +&347391\,{y}^{12}+561997\,{y}^{13}+882742\,{y}^{14}+1348938\,{y}^{15}+2010779\,{y}^{16}+2928154\,{y}^{17}+4173803\,{y}^{18}+5830039\,{y}^{19} \\
  +&7992068\,{y}^{20}+10761439\,{y}^{21}+14249943\,{y}^{22}+18568515\,{y}^{23}+23832260\,{y}^{24}+30144300\,{y}^{25}+37602936\,{y}^{26}\\
   +&46279969\,{y}^{27}+56232269\,{y}^{28}+67474300\,{y}^{29}+79996493\,{y}^{30}+93732893\,{y}^{31}+108588539\,{y}^{32}+124402462\,{y}^{33}\\
   +&140987325\,{y}^{34}+158087400\,{y}^{35}+175431907\,{y}^{36}+192688207\,{y}^{37}+209528800\,{y}^{38}+225578607\,{y}^{39}+240495470\,{y}^{40}\\
   +&253908861\,{y}^{41}+265511859\,{y}^{42}+274990182\,{y}^{43}+282121684\,{y}^{44}+286693656\,{y}^{45}+288607900\,{y}^{46}+287784490\,{y}^{47}\\
   +&284268896\,{y}^{48}+278123786\,{y}^{49}+269535199\,{y}^{50}+258694438\,{y}^{51}+245902795\,{y}^{52}+231445735\,{y}^{53}+215696263\,{y}^{54}\\
   +&198987144\,{y}^{55}+181711391\,{y}^{56}+164197341\,{y}^{57}+146808026\,{y}^{58}+129822649\,{y}^{59}+113534001\,{y}^{60}+98140965\,{y}^{61}\\
   +&83842462{y}^{62}+70743178{y}^{63}+58943269{y}^{64}+48456737{y}^{65}+39295487{y}^{66}+31400920{y}^{67}+24718505{y}^{68}\\+&19141614{y}^{69}
   +14576321\,{y}^{70}+10894761\,{y}^{71}+7988749\,{y}^{72}+5731894\,{y}^{73}+4021600\,{y}^{74}+2748644\,{y}^{75}+1828595\,{y}^{76}\\+&1177014\,{y}^{77}
   +732167\,{y}^{78}+435841\,{y}^{79}+247875\,{y}^{80}+132180\,{y}^{81}+65960\,{y}^{82}\\
   +&29549\,{y}^{83}+11862\,{y}^{84}+3740\,{y}^{85}+897\,{y}^{86}
.\end{align*}
}
By using the above basic ideas, we can reproduce $G^*(y)$ in 3 minutes.
%
%\begin{CJK}{GBK}{song}
%
%
%magic cube ÔÚÈº×÷ÓÃÏÂÎÒÃÇ¿ÉÒÔ¼ÙÉèx[1]×îÐ¡
%ÔÚx[1] ±£³Ö²»¶¯ÏÂµÄ×ÓÈº£¬Ëü×÷ÓÃÔÚ±ß¼¯µÄ¹ìµÀ £¨1£©£¨8£©£¨3£¬9£¬2£¬10£©£¨4£¬5£©(6,7,11,12)
%ÎÒÃÇ¶ÔÃ¿¸ö¹ìµÀ´æÔÚ×îÐ¡Öµ, ÆäÖÐ 8 ËùÔÚµÄ¹ìµÀÒÔ¼°£¨2£¬3£¬9£¬10£©ËùÔÚµÄ¹ìµÀ²»¿ÉÄÜº¬ÓÐÊÇ×îÐ¡ÖµµÄ£¬
%ÎÒÃÇÓÐÁ½ÖÖÇé¿ö ¹ìµÀ£¨4£¬5£© »òÕß£¨6£¬7£¬11£¬12£©ÖÐº¬ÓÐ×îÐ¡Öµ
%
%Çé¿öÒ»£º 6 ÊÇ×îÐ¡Öµ
%
%         ÔÚ 1£¬6 ±£³Ö²»¶¯ÏÂµÄ×ÓÈº ÖÐÖ»ÓÐÒ»¸öµ¥Î»Ôª£¬ÎÒÃÇ¿ÉÒÔÓÐÏÂÃæ È¡³£ÊýÏî½á¹û
%
%Çé¿ö¶þ£º 4ÊÇ×îÐ¡Öµ
%        ÔÚ1£¬4 ±£³Ö²»¶¯ÏÂµÄ×ÓÈº£¬ Ëü×÷ÓÃÔÚ±ß¼¯µÄ¹ìµÀ £¨1£©£¨4£©£¨5£©£¨8£©£¬£¨2£¬3£©£¨6£¬7£©£¨9£¬10£©£¨11£¬12£©£¬ÕâÊ±ÎÒÃÇÖ»ÐèÏÞÖÆ 6 < 7 ,¿ÉÒÔ±£Ö¤Î¨Ò»£¬  ÎÒÃÇ¿ÉÒÔÓÐÈ¡³£ÊýÏî½á¹û
%\end{CJK}

%For example, if we want to count nonnegative integral
%solutions of the linear equation $a_1 + a_2 - a_3 = 0$, we can simply write the generating function as
%
%$$
%\sum_{\substack{ a_1,a_2,a_3\geq 0 \\ a_1+a_1-a_3=0}}x_{1}^{a_1}x_{2}^{a_2}x_{3}^{a_3}=\sum_{a_1,a_2,a_3\geq 0} \CT_{\lambda} \lambda^{a_1+a_1-a_3}x_{1}^{a_1}x_{2}^{a_2}x_{3}^{a_3} =\CT_{\lambda} \frac{1}{(1-\lambda x_1)( 1- \lambda x_2)(1-x_3/ \lambda)}
%$$

%Using a computer we can easily obtain:

\section{Summary\label{s-summary}}

In this paper, we give a complete construction of all magic labellings of the cube.
They are classified by eight types. We used them to enumerate magic distinct labellings of the cube.
The computation is significantly simplified when the group of the symmetry encountered.
The method will apply to much larger graphs with good group of symmetry.


\begin{thebibliography}{99}
\bibitem{Ahmed-PhD-thesis}
M. Ahmed, \newblock Algebraic Combinatorics of Magic Squares, Ph.D. thesis, University of California, Davis, 2004, arXiv:math.CO/0405476.

%\bibitem{MacM-par-analy}
%G. E. Andrews, P. Paule, A. Riese, and V. Strehl, \newblock MacMahon's partition analysis.
%V. Bijections, recursions, and magic squares, Algebraic combinatorics and applications, Springer, Berlin, 2001, pp. 1-39.

\bibitem{Andrews-Omege-1}
G. E. Andrews, MacMahon¡¯s partition analysis. I. The lecture hall partition theorem, Mathematical essays in honor of Gian-Carlo Rota (Cambridge, MA, 1996), Progr. Math., vol.161, Birkh\"{a}user Boston, Boston, MA, 1998, pp. 1--22.


\bibitem{Andrews-Omege}
G. E. Andrews, P. Paule, and A. Riese, MacMahons partition analysis: the Omega package, European J. Combin. 22(2001).



\bibitem{Stanley-vol-1}
R. P. Stanley, \newblock Enumerative Combinatorics, Volume I, Cambridge, 1997.

\bibitem{Stanley-Linear-M_G}
R. P. Stanley, \newblock Linear homogeneous diophantine equations and magic labelings of graphs, Duke Mathematical Journal, Vol. 40, September 1973, 607-632.


\bibitem{Stanley-M_G-S_M_G}
R. P. Stanley, \newblock Magic Labelings of Graphs, Symmetric Magic Squares, Systems of Parameters and Cohen-Macaulay Rings, Duke Mathematical Journal, Vol. 43, No.3, September 1976, 511-531.


\bibitem{Stewart-M_G}
B. M. Stewart, \newblock Magic graphs, Canad. J. Math., Vol. 18, (1966), 1031-1059.


\bibitem{Stewart-S_c_g}
B. M. Stewart, \newblock Supermagic complete graphs, Canad. J. Math., Vol. 19, (1967), 427-438.

\bibitem{Xin-three-order}
Guoce Xin. Constructing all magic squares of order three. Discrete Mathematics, Vol. 308(15), (2008), 3393-3398.


\bibitem{Xin-fast-alg}
G. Xin, \newblock A fast algorithm for MacMahon's partition analysis, Electron. J. Combin., Vol. 11 (2004), R58, 20 pp. (electronic).

\bibitem{Xin-euclid-alg}
G. Xin, \newblock A Euclid style algorithm for MacMahon¡¯s partition analysis, J. Combin. Theory A., Vol. 131, (2015), 32-60.


\bibitem{Zhong}
G. Xin, X. Xu, C. Zhang and Y. Zhong, \newblock On magic distinct labellings of simple graphs, arXiv:2107.03161.
%\bibitem{Mazin}
%Mikhail Mazin, A bijective proof of Loehr-Warrington's formulas for the statistics ctot$_{\frac q p}$ and mid$_{\frac q p}$, Annals Combin., 18:709--722, 2014.
%
%\bibitem{Mellit}
%A. Mellit, Toric braids and $(m, n)$-parking functions, preprint (2016), arXiv:1604.07456.




\end{thebibliography}
\end{document}